\documentclass[11pt,a4paper,reqno]{amsart}
\usepackage{amscd,amsmath}
\usepackage{amssymb,amsmath,latexsym,color,enumitem,tikz,dsfont,tikz-cd}
\usepackage[all]{xypic}       
\usepackage[varg]{pxfonts}
\usepackage{tabu,adjustbox,booktabs}
\usepackage{quiver}
\usepackage{graphicx}
 \usepackage{comment}
\usepackage{hyperref}

\usetikzlibrary{decorations.pathmorphing,shapes}

\usetikzlibrary{decorations.markings}
\tikzset{negated/.style={
        decoration={markings,
            mark= at position 0.5 with {
                \node[transform shape] (tempnode) {$\backslash$};
            }
        },
        postaction={decorate}
    }
}

 \newcommand{\newuparrow}{{{\rlap{$\ $}\hbox{$\uparrow$}}}}
 \newcommand{\twoheaddownarrow}{{\rlap{\rlap{$\ $}\raise .25ex\hbox{$\downarrow$}}\raise-.25ex\hbox{$\downarrow$}}}
 \newcommand{\twoheaduparrow}{{\rlap{\rlap{$\ $}\raise .25ex\hbox{$\uparrow$}}\raise-.25ex\hbox{$\uparrow$}}}

\newcommand{\set}[1]{\{\,#1\,\}}
\newcommand{\bigset}[1]{\bigl\{\,#1\,\bigr\}}
\newcommand{\tbigwedge}{\mathop{\textstyle \bigwedge }}
\newcommand{\tbigsqcap}{\mathop{\textstyle \bigsqcap }}
\newcommand{\tbigcap}{\mathop{\textstyle \bigcap }}
\newcommand{\tbigcup}{\mathop{\textstyle \bigcup }}

\newcommand{\tbigvee}{\mathop{\textstyle \bigvee }}
 \newcommand{\cat}[1]{\ensuremath{\mathsf{#1}}} 
 
\makeatletter
\newcommand*{\@old@slash}{}\let\@old@slash\slash
\def\slash{\relax\ifmmode\delimiter"502F30E\mathopen{}\else\@old@slash\fi}
\makeatother

\def\SS{\mathsf{S}}

\def\cl{\mathfrak{c}}
\def\op{\mathfrak{o}}

\newtheorem{theorem}{Theorem}[section]
\newtheorem{proposition}[theorem]{Proposition}

\newtheorem{lemma}[theorem]{Lemma}
\newtheorem{corollary}[theorem]{Corollary}

\theoremstyle{definition}

\newtheorem{example}[theorem]{Example}

\theoremstyle{remark}

\newcommand{\ca}[1]{\mathcal{#1}}

\newcommand{\mf}{\mathsf}

\newcommand{\mi}{\mathit}

\newcommand{\se}{\subseteq}

\newcommand{\we}{\wedge}
\newcommand{\ve}{\vee}

\newcommand{\bca}{\bigcap}

\newcommand{\sll}{\mf{S}(L)}
\newcommand{\So}{\mf{S}_{\op}}
\newcommand{\Sc}{\mf{S}_{\cl}}
\newcommand{\Sb}{\mf{S}_b}

\newcommand{\ra}{\rightarrow}

\newcommand{\up}{{\uparrow}}

\begin{document}
\title[The lattice of smooth sublocales as a Bruns--Lakser completion]{The lattice of smooth sublocales as a Bruns--Lakser completion}
 \author[I.~Arrieta \and A.~L.~Suarez]
{Igor~Arrieta* \and Anna Laura Suarez**}

\newcommand{\acr}{\newline\indent}

\address{\llap{*\,}Departamento de Matem\'aticas\acr
Universidad del Pa\'{\i}s Vasco UPV/EHU\acr
48080 Bilbao\acr
SPAIN}
\email{igor.arrieta@ehu.eus}

\address{\llap{**\,}Department of Mathematics and Applied Mathematics\acr
University of the Western Cape\acr
Private Bag X17 Bellville 7535, Cape Town\acr
SOUTH AFRICA}
\email{annalaurasuarez993@gmail.com}

\thanks{\emph{Acknowledgement.} The results of this work were obtained during a visit of the second-named author to the first-named author, supported by the Basque Government grants IT1483-22 and IT1913-26. The first-named author also acknowledges support from the Basque Government through grants IT1483-22 and IT1913-26, and a postdoctoral fellowship (grant POS-2022-1-0015 and POS-2025-2-0019).}

\begin{abstract}
    We characterise the frame morphisms $f:L\to M$ that lift to frame maps $\overline{f}:\Sb(L)\to \Sb(M)$, where $\Sb(L)$ is the collection of joins of complemented sublocales of a frame $L$, or equivalently the Booleanization of the collection $\sll$ of all its sublocales. We do so by proving that $\Sb(L)$ is isomorphic to the Bruns--Lakser completion of the meet-semilattice formed by the locally closed sublocales, i.e. the sublocales of the form $\cl(a)\cap \op(b)$ for $a,b\in L$.
\end{abstract}

\maketitle

\tableofcontents

\section{Introduction}

For a locale $L$, the collection $\SS_b(L)$ consisting of joins of locally closed sublocales (that is, those sublocales of the form $\tbigvee_{a\in A,b\in B} \cl(a)\cap\op(b)$ for $A,B\subseteq L$) is a complete Boolean algebra  (see \cite{Arrieta2022}). This was shown in \cite{BGJ2013} to be isomorphic to a construction by Funayama introduced in \cite{funayama59}, sometimes referred to as the \emph{Funayama envelope}.
This collection has attracted attention in point-free topology (e.g., the naturality of the construction as a maximal essential extension \cite{ball2018}, its role as
a discretization of a locale for modeling non-continuous localic maps \cite{DISCB}, its role as the $T_D$--hull of a frame \cite{bezhanishvili2025td}, etc). 
Closely related to this collection there is the collection $\SS_c(L)$ of joins of closed sublocales \cite{PicadoPultrTozzi2019}, which sits inside $\SS_b(L)$ as a subframe. When $L$ is subfit, and only in that case, the collections $\SS_c(L)$ and $\SS_b(L)$ coincide.

Neither of the constructions $L\mapsto \SS_b(L)$ or $L\mapsto \SS_c(L)$ is functorial in the category of frames \cite{ball2019,Arrieta2022} --- see also \cite[Section~7.7]{ArrietaPhD2022} for some positive results.

In this paper we characterise the frame morphisms that lift with respect to the $\Sb(-)$ construction. The characterization will follow from the main result in our paper: if we denote by $\SS_{lc}(L)$ the meet-semilattice of locally closed sublocales, then inclusion $\mf{S}_{lc}^{op}(L)\subseteq \SS_b(L)$ is, up to isomorphism, the \emph{Bruns--Lakser completion} of $\mf{S}_{lc}^{op}(L)$. Again, we note in parallel that $\cl:L\to \Sc(L)$, too, is the Bruns--Lakser completion of $L$, viewed as a join-semilattice.

Given the strong analogy between the $\Sb(-)$ and the $\Sc(-)$ construction, here and in our work we will illustrate the results with reference to this. A frame morphism $f:L\to M$ lifting to the $\Sc(-)$ construction means there being a frame map $\overline{f}:\Sc(L)\to \Sc(M)$ with $\overline{f}(\cl(x))=\cl(f(x))$ for all $x\in L$. A morphism $f$ lifts if and only if the following well-definedness condition holds: 
\begin{equation*}
\cl(x)\se \tbigvee_i \cl(x_i)\text{ implies }\cl(f(x)) \se \tbigvee_i \cl(f(x_i)).   \tag{WDc} \label{e: wdc}
\end{equation*}
 This is an instance of a much more general fact. The \emph{Bruns--Lakser completion} \cite{BrunsLakser1970} of a join-semilattice $S$ is the embedding $\up:S\to \ca{AU}(S)$ into the \emph{admissible upper sets}, namely those closed under admissible meets (i.e. the existing meets that distribute over all binary joins). The Bruns--Lakser completion is always a frame, and we can characterise maps $f:S\to T$ of join-semilattices that can be lifted to frame homomorphisms $\overline{f}\colon\ca{AU}(S)\to\ca{AU}(T)$. For an upper set $U$, we call $\ca{A}(U)$ the smallest admissible upper set containing it. This is equivalent to the general well-definedness condition
\begin{equation*}
x\in \ca{A} (\tbigcup_i U_i)\text{ implies }f(x)\in \ca{A}(\tbigcup_i \up f[U_i]).   \tag{WD} \label{e: wd'}
\end{equation*}

It is known that the morphisms that do lift are precisely the so-called admissible morphisms.

If the join-semilattice under consideration is a frame, admissible meets are called \emph{exact}\footnote{For the purpose of this paper, and in agreement with the work in \cite{moshier22}, we can think of exactness of a meet $\bigwedge_i x_i\in L$ as being the property that $\cl(\bigwedge_i x_i)=\bigvee_i \cl(x_i)$, which in a a frame is equivalent to admissibility.} and admissible morphisms, too, are called exact. It is known that $\cl:L\to \Sc(L)$ is the Bruns--Lakser completion of $L$. The condition $\cl(x)\se \tbigvee_i \cl(x_i)$ is equivalent to $x\in \ca{A}(\up \{x_i\mid i\in I\})$; \eqref{e: wdc} is then seen as a special case of \eqref{e: wd'}. Indeed, \eqref{e: wdc} is equivalent to having that $f:L\to M$ is an exact morphism.

In this work, we show that the problem of lifting morphisms to the $\Sb(L)$ construction is another special instance of the general result. 
Here, a morphism lifts if there is a frame map $\overline{f}$ making the following commute.
\[
\begin{tikzcd}
	{\SS_b(L)} & {\SS_b(M)} \\
	L & M
	\arrow["{{\overline{f}}}", from=1-1, to=1-2]
	\arrow["{{\op_L}}"', from=2-1, to=1-1]
	\arrow["f", from=2-1, to=2-2]
	\arrow["{{\op_M}}", from=2-2, to=1-2]
\end{tikzcd}
\label{square that has to commute}
\]

In this case, too, there being such a frame map amounts to a well-definedness condition:

\begin{equation*}
\cl(x)\cap \op(y)\se \tbigvee_i \cl(x_i)\cap \op(y_i) \text{ implies }\cl(f(x))\cap \op(f(y))\se \tbigvee_i \cl(f(x_i))\cap \op(f(y_i)).   \tag{WDb} \label{e: wdb}
\end{equation*}

In this work, we exhibit this as another instance of \eqref{e: wd'}. We first define a collection $\mf{LC}(L)\se L\times L^{op}$, anti-isomorphic to the meet-semilattice $\mf{S}_{lc}(L)$ of locally closed sublocales. We will then show that the admissible meets of $\mf{LC}(L)$ are exactly the joins of locally closed sublocales which are themselves locally closed. We will deduce that $\mf{LC}(L)\se \Sb(L)$ is isomorphic to the Bruns--Lakser completion of $\mf{LC}(L)$. We will prove as a consequence of this that the frame maps $f:L\to M$ that lift coincide with those that determine maps $\mf{LC}(f):\mf{LC}(L)\to \mf{LC}(M)$ of join-semilattices which are admissible (where $\mf{LC}(f):\cl(x)\cap \op(y)\mapsto \cl(f(x))\cap \op(f(y))$). We will also characterise such maps explicitly. 

We can even compare the situations for $\Sc(L)$ and $\Sb(L)$ with those for $\mf{S}_o(L)$, the collection of intersections of open sublocales, and $\mf{S}(L)$, the collection of all sublocales. In both cases all frame maps lift. For intersections of open sublocales, the well-definedness condition becomes:

\begin{equation*}
\bca_i \op(x_i)\se \op(x)\text{ implies }\bca_i \op(f(x_i))\se \op(f(x)).   \tag{WDo} \label{e: wdo}
\end{equation*}
This is equivalent to the frame map preserving all the meets $\bigwedge_i x_i\in L$ with $\bca_i \op(x_i)=\op(\bigwedge_i x_i)$. But these are known to be the \emph{strongly exact} meets, also known as \emph{free} meets (see \cite{wilson94}), and are well known to be preserved by any frame map, see \cite[3.5]{notesball}.
For $f:L\to M$ to lift to a coframe map $f:\mf{S}(L)\to \mf{S}(M)$, it is enough for the well-definedness condition below to hold.

\begin{equation*}
\bca_i \cl(x_i)\ve \op(y_i)\se \cl(x)\ve \op(y)\text{ implies }\bca_i \op(f(x_i))\ve \cl(f(y_i))\se \op(f(x))\ve \cl(f(y)).   \tag{WDs} \label{e: wds}
\end{equation*}

This is again equivalent to $\mf{LC}(f):\mf{LC}(L)\to \mf{LC}(M)$ preserving those intersections of locally closed sublocales which are locally closed, a condition satisfied by every frame map. What is missing is a characterisation of such intersections as joins of $\mf{LC}(L)$, in the same vein as the characterisation of strongly exact meets as those preserved by all frame maps, or even a purely algebraic characterisation as the one for locally closed meets of $\mf{LC}(L)$.

We have already highlighted the analogies between the problems of lifting morphisms for the $\mf{S}(L)$, $\So(L)$, $\Sc(L)$, and $\Sb(L)$ constructions. In light of these, it is natural to ask how all these can be objects of a more general study. Our proofs in this work go in the direction of this more general study, which will be the topic of an upcoming paper.

The paper is organised as follows. Section~\ref{sec.prel} collects the necessary preliminaries. In Section~\ref{sec.lc} we introduce the collection $\mathsf{LC}(L)\subseteq L\times L^{op}$ that is anti-isomorphic to the meet-semilattice of locally closed sublocales and establish its basic properties. In Section~\ref{sec.4} we characterise admissible meets $\tbigwedge{(a_i,b_i)}_{i\in I}$ in $\mathsf{LC}(L)$ as those for which $\bigvee_{i\in I}\cl(a_i)\cap\op(b_i)$ is locally closed. Section~\ref{sec.5} identifies $\SS_b(L)$ with the Bruns–Lakser completion of $\mathsf{LC}(L)$. Finally, Section~\ref{sec.lifts} characterises those maps between join-semilattices that admit a lifting to the Bruns–Lakser completion. As a by-product, we characterise the frame homomorphisms that lift to the construction $\SS_b(-)$.

\section{Preliminaries}\label{sec.prel}
Our notation and terminology regarding the categories of frames and locales will be that of \cite{PP12} (see also \cite{STONE}). The Heyting operator in a frame $L$, right adjoint to the meet operator, will be denoted by $\to$; for each $a\in L$, $a^*=a\to 0$ is the pseudocomplement of $a$. 

\subsection{Some Heyting rules}
\addtocontents{toc}{\protect\setcounter{tocdepth}{-1}}
For the reader's convenience, we list here some of the properties satisfied by the Heyting operator in a frame $L$. For any $a,b,c\in L$ and any $\{a_i\}_{i\in I}\subseteq L$, the following hold:
\begin{enumerate}[label=\textup{(H\arabic*)},leftmargin=2.0\parindent]
\item \label{H1} $1\to a=a$\textup;
\item \label{H2} $a\leq b$ if and only if $a\to b=1$\textup;
\item \label{H3} $a\leq b\to a$\textup;
\item \label{H4} $a\to b=a\to (a\wedge b)$\textup;
\item \label{H5} $a\wedge (a\to b)=a\wedge b$\textup;
\item \label{H6} $a\wedge b=a\wedge c$ if and only if $a\to b=a\to c$\textup;
\item \label{H7} $(a\wedge b)\to c=a\to (b\to c)= b\to (a\to c)$\textup;
\item \label{H8} $a=(a\vee b)\wedge (b\to a)$\textup;
\item \label{H9} $a\leq (a\to b)\to b$\textup;
\item \label{H10} $((a\to b)\to b)\to b=a\to b$\textup;
\item \label{H11}$(\tbigvee_{i\in I}a_i)\to b= \tbigwedge_{i\in I}(a_i\to b)$\textup;
\item \label{H12}$b\to (\tbigwedge_{i\in I}a_i)= \tbigwedge_{i\in I}(b\to a_i)$\textup.
\end{enumerate}

\addtocontents{toc}{\protect\setcounter{tocdepth}{2}}

\subsection{Sublocales}\label{subsect.sublocales}
\addtocontents{toc}{\protect\setcounter{tocdepth}{-1}}

A \emph{sublocale} of a locale $L$ is a subset $S\subseteq L$ closed under arbitrary meets such that
\[\forall a\in L,\ \ \ \forall s\in S,\  \ \ a\to s\in S.\]
These are precisely the subsets of $L$ for which the embedding $j_S\colon S\hookrightarrow L$ is a morphism of locales. Sublocales of $L$ are in one-to-one correspondence with the regular subobjects (equivalently, extremal subobjects) of $L$ in the category of locales. We denote by $\nu_S$  the associated frame surjection
given by
$$\nu_S(a)=\tbigwedge\{s\in S\mid a\leq s\}.$$
The system $\SS  (L)$ of all sublocales of $L$, partially ordered by inclusion, is a coframe \cite[Theorem~III.3.2.1]{PP12}, that is, its dual lattice is a frame.  Infima and suprema are given by
\[
\tbigwedge_{i\in I}S_i=\tbigcap_{i\in I}S_i, \quad \tbigvee_{i\in I}S_i=\set{\tbigwedge M\mid M\subseteq\tbigcup_{i\in I} S_i}.
\]
The least element is the sublocale $\mathsf{O}=\{1\}$ and the greatest element is the entire locale $L$. 

It is easy to see that given a collection $\{S_i\}_{i\in I}$ of sublocales then $$\nu_{\tbigvee_i S_i}(a)=\tbigwedge_i\nu_{S_i}(a).$$

For any $a\in L$, the sublocales
\[
\mathfrak{c}_L(a)=\newuparrow  a=\set{b\in L\mid b\ge a}\ \text{ and }\ \mathfrak{o}_L(a)=\set{a\to b\mid b\in L}
\]
are the \emph{closed} and \emph{open} sublocales of $L$, respectively (that we shall denote simply by $\mathfrak{c}(a)$ and $\mathfrak{o}(a)$ when there is no danger of confusion). For each $a\in L$, $\mathfrak{c}(a)$ and $\mathfrak{o}(a)$ are
complements of each other in $\SS(L)$
and satisfy the expected identities
\begin{equation*}{\label{identities.basic}}
\tbigcap_{i\in I} \mathfrak{c}(a_i)=\cl(\tbigvee_{i\in I} a_i),\quad \cl(a)\vee\cl(b)=\cl(a\wedge b),
\end{equation*}
\[\tbigvee_{i\in I}\op(a_i)=\op(\tbigvee_{i\in I} a_i) \quad\mbox{ and }\quad \op(a)\cap \op(b)=\op(a\wedge b).
\]
A sublocale is \emph{locally closed} if it is of the form $\cl(a)\cap\op(b)$ for $a,b\in L$.
Moreover, $$\nu_{\cl(a)\cap\op(b)}(x)= b\to (a\vee x)$$
for all $x\in L$.

Given a sublocale $S$ of $L$, its \emph{closure}, denoted by $\overline{S}$, is the smallest closed sublocale containing it. In this context, the formula $\overline{S}=\cl(\tbigwedge S)$ holds.

Moreover, $\SS(L)$ is zero-dimensional in the sense that every sublocale $S\subseteq L$ can be expressed as $S=\tbigcap_{a\in A,\ b\in B}\op(a)\ve\cl(b)$ for some subsets $A,B\subseteq L$ (see \cite[Proposition~III.6.5]{PP12}).
\addtocontents{toc}{\protect\setcounter{tocdepth}{2}}

\subsection{The Boolean algebra of smooth sublocales.}\label{sb_prelim}
\addtocontents{toc}{\protect\setcounter{tocdepth}{-1}}

We  will write $S^{\#}$ for the supplement (i.e. co-pseudocomplement) of a sublocale $S\subseteq L$. A sublocale $S$ is said to be \emph{smooth} if it is of the form $S=T^{\#}$ for some sublocale $T\subseteq L$, equivalently if $S^{\#\#}=S$. By the zero-dimensionality of $\SS(L)$ it follows that a sublocale is smooth if and only if it is a join of complemented sublocales, if and only if it is a join of locally closed sublocales.
We denote by $\SS_b(L)$ the subset of $\SS(L)$ consisting of smooth sublocales, i.e.
$$\SS_b(L)=\bigset{ \tbigvee_{a\in A,b\in B} \cl(a)\cap \op(b)\mid A,B \subseteq L}.$$

The collection $\SS_b(L)$ is a subcolocale of $\SS(L)$, which is a Boolean algebra, and it is in fact the Booleanization of $\SS(L)$ --- see \cite{Arrieta2022}
 for details.
\addtocontents{toc}{\protect\setcounter{tocdepth}{2}}

 \subsection{The Bruns--Lakser completion of a join-semilattice}
 \addtocontents{toc}{\protect\setcounter{tocdepth}{-1}}

\label{brunslakser}
In what follows we briefly outline the construction of a slight adaptation for join-semilattices of the Bruns--Lakser completion, introduced in \cite{BrunsLakser1970} for meet-semilattices as an explicit description of their injective hulls. Let $S$ be a join-semilattice. A family $\{a_i\}_{i\in I}\subseteq S$ is \emph{admissible} if its meet in $S$ exists and for each $b\in S$, the  equality  
$b\vee \tbigwedge_{i\in I}a_i = \tbigwedge_{i\in I} b\vee a_i$ holds. In this case, we also say that the meet $\tbigwedge_{i\in I} a_i$ is \emph{admissible}.

A nonempty upper set $U$ of $S$ is said to be \emph{admissible} if it is closed under admissible meets --- i.e., whenever $\{a_i\}_{i\in I}\subseteq U$ is admissible, one has $\tbigwedge_{i\in I}a_i\in U$. The collection of all nonempty upper sets (resp. all admissible upper sets) of $S$ is noted by $\mathcal{U}(S)$ (resp. $\mathcal{AU}(S)$). From the main result in \cite{BrunsLakser1970} it follows that  $\mathcal{AU}(S)$ is a frame.  We call the natural antitone embedding  $\uparrow\colon S\to \mathcal{AU}(S)$ the \emph{Bruns--Lakser completion} of $S${\footnote{ In \cite{BrunsLakser1970} the construction is described for meet-semilattices, and the completion consists of the downsets closed under admissible joins.}.

Morphisms of join-semilattices that lift to the Bruns--Lakser completion are the \emph{admissible} morphisms (cf. Theorem~\ref{t: characterize those that lift general} below), namely those morphisms $f:S\to T$ such that for any admissible meet $\bigwedge_i a_i\in S$ both $\bigwedge_i f(a_i)$ is admissible and $\bigwedge_i f(a_i)=f(\bigwedge_i a_i)$. These two conditions are independent. 
\addtocontents{toc}{\protect\setcounter{tocdepth}{2}}

 \subsection{Joins of closed sublocales and the Bruns--Lakser completion}\label{scl.BL}
 \addtocontents{toc}{\protect\setcounter{tocdepth}{-1}}

Let $L$ be a frame. We write $\SS_c(L)$ for the collection of all joins of closed sublocales of $L$, namely
\[
\SS_c(L)=\bigset{ \tbigvee_{a\in A} \cl(a)\mid A \subseteq L}.
\]
We consider $\SS_c(L)$ as an ordered set with the inclusion order inherited from $\SS(L)$.  
Picado, Pultr, and Tozzi~\cite{PicadoPultrTozzi2019} show that $\SS_c(L)$ is always a frame and that it embeds as a complete join-sublattice of the coframe $\SS(L)$. Moreover, their main result states that $\SS_c(L)$ is a Boolean algebra if and only if $L$ is subfit; in this case, $\SS_c(L)$ coincides with the family $\SS_b(L)$ introduced in Subsection~\ref{sb_prelim}.  

The connection between $\SS_c(L)$ and the Bruns--Lakser construction (see Subsection~\ref{brunslakser}) was established in~\cite{ballmoshierpultr}. More precisely,  the Bruns--Lakser completion of $L$, viewed as a complete join-semilattice, is isomorphic to $\SS_c(L)$.  The proof strategy in \cite{ballmoshierpultr} is as follows. One defines monotone maps $\varphi\colon \SS_c(L)\to \mathcal{U}(L)$ given by $\varphi(S):=\{a\in L\mid \cl(a)\subseteq S\}$ and $\psi\colon \mathcal{U}(L)\to \SS_c(L)$ given by $\psi(U)= \tbigvee_{a\in U}\cl(a)$.
Then, one shows that there is an adjunction $\psi\dashv
 \varphi$ which restricts to an isomorphism $\SS_c(L)\cong \ca{AU}(L)$ as the following diagram shows:
 
\begin{equation}\label{l: diagram with adjunction for ScL and upsets}
\begin{tikzcd}
	{\SS_c(L)} & \top & {\mathcal{U}(L)} \\
	\\
	{\SS_c(L)} && {\mathcal{AU}(L)}
	\arrow["\varphi", curve={height=-6pt}, from=1-1, to=1-3]
	\arrow["\psi", curve={height=-6pt}, from=1-3, to=1-1]
	\arrow[from=3-1, to=1-1,equals]
	\arrow["\cong"', from=3-1, to=3-3]
	\arrow[hook, from=3-3, to=1-3]
\end{tikzcd}    
\end{equation}

Note that, since every frame is a distributive lattice, in this case, finite meets are exact, and hence exact upper sets are actually filters of $L$ (referred to as \emph{exact filters} in \cite{ballmoshierpultr}). 
Frame morphisms $f:L\to M$ that are admissible (viewing $L$ and $M$ as join-semilattices) are called \emph{exact}. The frame morphisms that lift to this construction are precisely the exact morphisms (\cite[Proposition~6.6]{SuarezRaneyExtensions2025}).

Exact families have the following useful description.
\begin{lemma}\label{exactiffclosed}{\textup(\cite[Theorem~4.3]{notesball}\textup)\textup.}
If $L$ is a frame, a family $\{a_i\}_{i\in I}\subseteq L$ is exact if and only if the sublocale $\tbigvee_{i\in I}\cl(a_i)$ is closed.
\end{lemma}

\addtocontents{toc}{\protect\setcounter{tocdepth}{2}}

 \section{The join-semilattice of locally closed sublocales}\label{sec.lc}

Let $L$ be a frame. In this section we construct a join-semilattice that serves, in the context of smooth sublocales, as the appropriate analogue of $L$ (viewed as a join-semilattice) in the context of joins of closed sublocales (cf. Subsection~\ref{scl.BL}). 
For reasons that will become clear with the next lemma, for a locally closed sublocale $S$ we call the pair $(\tbigwedge S,\nu_{S^\#}(\tbigwedge S))$ its \emph{canonical representation}.

\begin{lemma}\label{lemmalocallyclosed1}Let $L$ be a frame and $S\subseteq L$ a sublocale. 
\begin{enumerate}
\item\label{lemmalocallyclosed11} One has $$\cl(\tbigwedge S)\cap \op(\nu_{S^\#}(\tbigwedge S))\subseteq S^{\#\#}.$$
\item\label{lemmalocallyclosed12}  Moreover, $S$ is locally closed if and only if $S\subseteq \op(\nu_{S^\#}(\tbigwedge S))$.  In that case, one has $S= \cl(\tbigwedge S)\cap \op(\nu_{S^\#}(\tbigwedge S))$.
\item\label{lemmalocallyclosed13}  If $S=\cl(a)\cap \op(b)$ and the relations $b\to a=a$ and $a\leq b$ hold, then  $(a,b)$ is the canonical representation of $S$. In other words, locally closed sublocales correspond exactly to pairs $(a,b)$ with $a\leq b$ and $b\to a = a$.
\end{enumerate}
\end{lemma}

\begin{proof}
(\ref{lemmalocallyclosed11}) Let $y\in \cl(\bigwedge S)$. Since $L=S^\#\vee S^{\#\#}$,  there  are $a\in S^\#$ and $b\in S^{\#\#}$ such that $y=a\wedge b$. Now since $\bigwedge S\leq y\leq b$, it follows that $x:= \nu_{S^\#}(\tbigwedge S) \leq \nu_{S^\#}(b)$, and since $\bigwedge S\leq y\leq a$, we  have $x\leq \nu_{S^\#}(a)=a$, where the last equality follows because $a\in S^\#$. Therefore $a\wedge \nu_{S^\#}(b)\geq x$. Hence we have $y= a\wedge b = a\wedge ( \nu_{S^\#}(b)\wedge b) =(a\wedge \nu_{S^\#}(b))\wedge b$: the first part of this meet belongs to $\cl(x)$ and the second part belongs to $S^{\#\#}$. We have thus shown $\cl(\bigwedge S)\subseteq \cl(x)\vee S^{\#\#}$. From this the required inclusion follows.\\[2mm]
(\ref{lemmalocallyclosed12}) If $S\subseteq \op(\nu_{S^\#}(\tbigwedge S))$, since we also have $S\subseteq \overline{S}= \cl(\tbigwedge S)$, by the first part it follows that  $S$ is locally closed.  Conversely, assume that $S$ is locally closed, i.e. $S=S^{\#\#}=\cl(a)\cap \op(b)$. Since $S\subseteq \overline{S}=\cl(\tbigwedge S)$, it only suffices to check that $S\subseteq \op(x)$. Note that $\tbigwedge S= b\to a$, and so $$x=\nu_{\op(a)\vee\cl(b)}(\tbigwedge S) =(a\to (b\to a))\wedge (b\vee(b\to a))= b\vee (b\to a)$$ by \ref{H2}, \ref{H3} and \ref{H7}. Then $S\subseteq \op(x)$ if and only if $\cl(a)\cap\op(b)\subseteq \op(b\vee (b\to a))$, and this holds if and only if $b\leq a\vee (b\to a)\vee b$, which clearly holds. Hence, $S=  \cl(\tbigwedge S)\cap \op(\nu_{S^\#}(\tbigwedge S))$.\\[2mm]
(\ref{lemmalocallyclosed13}) Let $S=\cl(a)\cap \op(b)$ with $a\leq b$ and $b\to a=a$. By the previous calculations $\tbigwedge S= b\to a = a$ and $\nu_{S^\#}(\tbigwedge S)=b\vee (b\to a)=b\vee a= b$. 
\end{proof}

In view of Lemma~\ref{lemmalocallyclosed1}\,(\ref{lemmalocallyclosed13}), we define
$$\cat{LC}(L):=\{(a,b)\in L\times L\mid a\leq b,\ b\to a=a\},$$
so that $\mf{LC}(L)$ is the collection of all canonical representations of locally closed sublocales. We want to equip this set with a partial order so that it is anti-isomorphic to the collection of locally closed sublocales. To that end, we make use of the following result.

\begin{lemma}\label{lemmaordering111}Let $L$ be a frame.
If $(a,b),(x,y)\in	\cat{LC}(L)$, then $\cl(a)\cap\op(b)\subseteq \cl(x)\cap\op(y)$ if and only if $x \leq a$ and $b\leq a\vee y$.
\end{lemma}

We now define a partial order $\sqsubseteq$ on $\mf{LC}(L)$, defined as $(x,y)\sqsubseteq(a,b)$ if and only if $x \leq a$ and $b\leq a\vee y$. Observe that Lemma~\ref{lemmalocallyclosed1}\,(3) guarantees that $\sqsubseteq$ is a partial order. In summary:
$$(x,y)\sqsubseteq (a,b) \iff (x\leq a \text{ and } b\leq a\vee y) \iff \cl(a)\cap \op(b)\subseteq \cl(x)\cap \op(y).$$

We shall always consider $\cat{LC}(L)$ with the ordering $\sqsubseteq$.

\begin{lemma}\label{l: bwe S for a smooth sublocale S}
Let $L$ be a frame and $S=\tbigvee_{i\in I}\cl(a_i)\cap \op(b_i)$. Then
$$\nu_{S^\#}(\tbigwedge S) = \tbigwedge_{x\in L} \biggl(  \big(\tbigwedge _{i\in I} (b_i\to (x\vee a_i)) \big) \to \big( x\vee  \tbigwedge_{j\in I}(b_j\to a_j)\bigr) \biggr).$$
\end{lemma}

\begin{proof}
Note that since $S^{\#}$ is smooth, we can write \begin{align*} S^{\#}&= \tbigvee\{ \cl(x)\cap\op(y)\mid \cl(x)\cap\op(y)\subseteq S^{\#} \}
\\& = \tbigvee\{ \cl(x)\cap\op(y)\mid \forall i\in I,\ \cl(a_i)\cap\op(b_i)\subseteq \op(x)\vee \cl(y) \}\\& = \tbigvee\{ \cl(x)\cap\op(y)\mid \forall i\in I,\ b_i\wedge y \leq x\vee a_i \}
\\&  =\tbigvee\{ \cl(x)\cap\op(\tbigwedge _i (b_i\to (x\vee a_i)))\mid x\in L\}.
\end{align*}
Moreover, since $\tbigwedge S= \tbigwedge_i (b_i\to a_i)$, we conclude that
$$\nu_{S^{\#}} (\tbigwedge S)=\tbigwedge_{x\in L} \biggl(  \big(\tbigwedge _i (b_i\to (x\vee a_i)) \big) \to \big( x\vee  \tbigwedge_j(b_j\to a_j)\bigr) \biggr).$$
\end{proof}

We define the map $\mi{lc}:L\times L\to L\times L$ as 
\[
\mi{lc}(a,b)=(b\ra a,(b\ra a)\vee b),
\]
and so $\mi{lc}(a,b)=(\tbigwedge S,\nu_{S^\#}(\tbigwedge S))$ where $S=\cl(a)\cap \op(b)$. Under the anti-isomorphism with locally closed sublocales, this is the same as $(a,b)\mapsto \cl(a)\cap \op(b)$. The anti-isomorphism between $\mf{LC}(L)$ and the locally closed sublocales readily gives us the following result.
\begin{lemma}\label{l: finite joins in LCL}\label{l: lc does not affect the joins}Let $L$ be a frame.
The poset $\mf{LC}(L)$ is a join-semilattice, with joins computed as $(x,y)\sqcup (u,v)=\mi{lc}(x\ve u,y\we v)$. Furthermore, $\mi{lc}(a,b)\sqcup \mi{lc}(x,y)=\mi{lc}(a\ve x,y\we b)$ for all $(a,b),(x,y)\in L\times L$.
\end{lemma}

\section{Admissible meets of locally closed sublocales}\label{sec.4}

In the previous section we introduced the poset $\mathsf{LC}(L)$ and showed that it is a join-semilattice.
The aim of this section is to study admissibility in $\mathsf{LC}(L)$ (see Subsection~\ref{brunslakser}). 
Motivated by Lemma~\ref{exactiffclosed}, which characterises exact families  in a frame $L$ as those families $\{a_i\}_{i\in I}\subseteq L$ for which the join $\tbigvee_{i\in I} \cl(a_i)$ is closed, we seek an analogous criterion for families of locally closed sublocales. 
Thus, we introduce the notion of {local exactness}: a family $\{(a_i,b_i)\}_{i\in I}\subseteq L\times L$  is \emph{locally exact} when the sublocale $\tbigvee_i \cl(a_i)\cap \op(b_i)$ is locally closed. By Lemmas~\ref{lemmalocallyclosed1} and \ref{lemmaordering111}, a family $\{(a_i,b_i)\}_{i\in I}$ is locally exact if and only if 
\[
b_i\leq \nu_{S^\#}(\tbigwedge S)\ve a_i\text{ for all $i\in I$,}
\]
where $S=\tbigvee_{i\in I} \cl(a_i)\cap \op(b_i)$. By the formula for $\nu_{S^\#}(\tbigwedge S)$ of Lemma~\ref{l: bwe S for a smooth sublocale S},  this amounts to having, for each $i\in I$, that:
$$b_i \leq a_i \vee  \tbigwedge_{x\in L} \biggl(  \big(\tbigwedge _{i\in I} (b_i\to (x\vee a_i)) \big) \to \big( x\vee  \tbigwedge_{j\in I}(b_j\to a_j)\bigr) \biggr).$$

\begin{example}
A family of the form $\{(a_i,1)\}_{i\in I}$ is locally exact if and only if for all $i\in I$
$$ a_i\vee \tbigwedge_{x\in L}[ (\tbigwedge_i (x\vee a_i))\to (x\vee \tbigwedge_j a_j) ]= 1;$$
So this is a sort of generalised exactness.
\end{example}

We now characterise the admissible meets in $\mf{LC}(L)$ as the meets of locally exact families. 

\begin{proposition}\label{p: exact iff preserved by psi}
    Let $S$ be a join-semilattice and let $\psi:S^{op}\to C$ be an injective meet-semilattice map such that $C$ is a coframe. Suppose that $\{a_i\}_{i\in I}\subseteq S$ is a collection whose meet exists in $S$.
    \begin{enumerate}
        \item If $\psi(\tbigwedge_{i\in I} a_i)=\tbigvee_i \psi(a_i)$, the meet $\tbigwedge_{i\in I} a_i$ is admissible.
        \item  \label{with complements generates}If, moreover, elements of the form $\psi(a)$ are complemented and every element of $C$ is of the form $\tbigwedge_{j\in J} \psi(u_j)\vee \psi(v_j)^*$, the converse holds, too.
    \end{enumerate}
\end{proposition}
\begin{proof}
The first claim follows from an argument analogous to Proposition~3.7.2 in \cite{moshier22}. We note that the assumption in \cite{moshier22} that $\psi(a)$ is complemented, thus colinear, here is replaced by the assumption that $C$ is a coframe, and so all its elements are colinear. For the second claim, suppose that for $\{a_i\}_i\subseteq S$ the meet $\tbigwedge_{i\in I} a_i$ exists in $S$ and is admissible. Towards contradiction, suppose that $\psi(\tbigwedge_i a_i)\nleq \tbigvee_i \psi(a_i)$. By assumption, there are $u,v\in S$ with $\psi(a_i)\leq \psi(u)\vee \psi(v)^*$ for all $i\in I$ and $\psi(\tbigwedge_i a_i)\nleq \psi(u)\vee \psi(v)^*$. Then, for all $i\in I$:
\[
\psi(a_i)\we \psi(v)\leq \psi(u)
\]
Using the fact that $\psi:S^{op}\to C$ is injective and preserves all binary meets, this means $u\leq a_i \vee v$. Using admissibility of $\tbigwedge_i a_i$, this implies $u\leq \tbigwedge a_i \vee v$. On the other hand, because $\psi(\tbigwedge_i a_i)\nleq \psi(u)\vee \psi(v)^*$, one has
\[
\psi(\tbigwedge_i a_i)\we \psi(v)\nleq \psi(u),
\]
from which we obtain $u\nleq \tbigwedge_i a_i \ve v$, using again that $\psi$ is an injective map of meet-semilattices. This is a contradiction.
\end{proof}

As a direct consequence of this, we obtain:
\begin{corollary}\label{t: locally exact meet iff exact meet}
Let $L$ be a frame and $\{ (a_i,b_i) \}_{i\in I}\subseteq \mathsf{LC}(L)$. Then $\{ (a_i,b_i) \}_{i\in I}$ is locally exact if and only if the family $\{(a_i,b_i)\}_{i\in I}$ is admissible in $\mathsf{LC}(L)$.    
\end{corollary}
\begin{proof}
    Consider the map $\mf{LC}(L)^{op}\to \mf{S}(L)$ defined as $(a,b)\mapsto \cl(a)\cap \op(b)$. This is a map of meet-semilattices, and, as we know, it is injective. It is easy to see, via basic facts about $\sll$, that the map satisfies the hypotheses of Proposition~\ref{p: exact iff preserved by psi}, including those of its Item~(\ref{with complements generates}), and so the claim follows from said proposition.
\end{proof}

\section{Smooth sublocales as admissible upper sets}\label{sec.5}

The aim of this section is to relate the Bruns--Lakser completion of the join-semilattice $\mathsf{LC}(L)$ to the lattice of smooth sublocales $\mathsf{S}_b(L)$ introduced in Subsection~\ref{sb_prelim}.  
This mirrors the well known situation in Subsection~\ref{scl.BL}, where exact upper sets (equivalently, filters) of $L$ correspond to joins of closed sublocales. First note that there is a monotone map $$\varphi\colon \SS_b(L)\longrightarrow \ca{U}(\mathsf{LC}(L))$$
to the nonempty upper sets of $\mf{LC}(L)$, given by $\varphi(S)=\{(a,b)\in \cat{LC}(L)\mid \cl(a)\cap\op(b)\subseteq S\}$ (cf. Subsection~Subsection~\ref{brunslakser}). Indeed, by Lemma~\ref{lemmaordering111} and the definition of the ordering $\sqsubseteq$ this is an upper set of $\cat{LC}(L)$. Moreover, there is another monotone map $$\psi\colon \ca{U}(\mathsf{LC}(L))\to\SS_b(L)$$ given by $$\psi(U)=\tbigvee_{(c,d)\in U}\cl(c)\cap\op(d).$$

Clearly, the adjunction identities

$$ \psi\varphi(S)\subseteq S \quad \text{and} \quad U\subseteq  \varphi\psi(U)$$

hold. Since $\varphi$ is clearly injective, $\psi \circ \varphi=\text{id}$.

Note that, by Corollary~\ref{t: locally exact meet iff exact meet} a nonempty upper set $U\subseteq \mathsf{LC}(L)$ is admissible (in the sense of Subsection~\ref{brunslakser}) if whenever $\{(a_i,b_i)\}_{i\in I}\subseteq U$ is locally exact, then the meet $\bigsqcap_{i\in I} (a_i,b_i)$ in $\mathsf{LC}(L)$  belongs to $U$. 
This is in turn equivalent to the relation
$$\left( \tbigwedge_{i\in I} a_i,    \tbigwedge_{x\in L} \biggl(  \big(\tbigwedge _{i\in I} (b_i\to (x\vee a_i)) \big) \to \big( x\vee  \tbigwedge_{j\in I} a_j\bigr) \biggr)\right)\in U$$
holding whenever $\{(a_i,b_i)\}_{i\in I}\subseteq U$ is a locally exact family. Indeed,  if $\{(a_i,b_i)\}_{i\in I}$ is locally exact, and we denote  $S=\tbigvee_i \cl(a_i)\cap \op(b_i)$ then by local exactness   one has $S=\cl(\tbigwedge S)\cap \op(\nu_{S^\#}(\tbigwedge S))$, and so   $\cl(\tbigwedge S)\cap \op(\nu_{S^\#}(\tbigwedge S))$ is the join of $\{\cl(a_i)\cap \op(b_i)\}_{i\in I}$ in  the meet-semilattice of locally closed sublocales. Using the anti-isomorphism between $\mf{LC}(L)$ and the meet-semilattice of locally closed sublocales we then have that the meet of $\{(a_i,b_i)\}_{i\in I}$ in $\mathsf{LC}(L)$ must be $(\tbigwedge S,\nu_{S^\#}(\tbigwedge S))$, and so by Lemma~\ref{l: bwe S for a smooth sublocale S}, indeed we have the following formula for the meet of $\{(a_i,b_i)\}_{i\in I}$
\begin{equation}\label{e: meet of locally exact family}
\tbigsqcap_{i\in I} (a_i,b_i)=\left( \tbigwedge_{i\in I} a_i,    \tbigwedge_{x\in L} \biggl(  \big(\tbigwedge _{i\in I} (b_i\to (x\vee a_i)) \big) \to \big( x\vee  \tbigwedge_{j\in I} a_j\bigr) \biggr)\right).
\end{equation}

Following the notation introduced in Subsection~\ref{brunslakser}, we denote the family of admissible upper sets by $$\ca{AU}(\mf{LC}(L)).$$

\begin{lemma}\label{l: fixpoints phipsi}
The relation $\varphi(\psi(U))=U$ holds if and only if $U$ is admissible.
\end{lemma}

\begin{proof}
Assume first that  $\varphi(\psi(U))=U$, then $U=\{ (a,b)\in \mathsf{LC}(L)\mid \cl(a)\cap\op(b)\subseteq \tbigvee_{(c,d)\in U}\cl(c)\cap\op(d)\}$. Now if $\{(a_i,b_i)\}\subseteq U$ is a locally exact family, then $\tbigvee_{i\in I}\cl(a_i)\cap\op(b_i)$ is locally closed and moreover it equals 
$$\cl(\tbigwedge_i a_i)\cap \op\biggl(   \tbigwedge_{x\in L} \biggl(  \big(\tbigwedge _i (b_i\to (x\vee a_i)) \big) \to \big( x\vee  \tbigwedge_j a_j\bigr) \biggr) \biggr),$$
so by assumption, it follows $$\left(\tbigwedge a_i,   \tbigwedge_{x\in L} \biggl(  \big(\tbigwedge _i (b_i\to (x\vee a_i)) \big) \to \big( x\vee  \tbigwedge_j a_j\bigr) \biggr)\right)\in U.$$

Let us now show the converse. Assume that $U$ is closed under locally exact families. We only need to show that $$\left\{ (a,b)\in \mathsf{LC}(L)\mid \cl(a)\cap\op(b)\subseteq \tbigvee_{(c,d)\in U}\cl(c)\cap\op(d)\right\}\subseteq U.$$
Let $(a,b)\in \mathsf{LC}(L)$ with $ \cl(a)\cap\op(b)\subseteq \tbigvee_{(c,d)\in U}\cl(c)\cap\op(d)\subseteq U$. 
Then
\begin{align*}\cl(a)\cap \op(b)&= \tbigvee_{(c,d)\in U}\cl(c\vee a)\cap\op(d\wedge b) \\ &=\tbigvee_{(c,d)\in U}\cl\bigl((d\wedge b)\to (c\vee a)\bigr)\cap\op\bigl((d\wedge b)\vee (d\wedge b)\to(c\vee a)\bigr)
\end{align*}
Note that $A:=\biggl\{ \bigl((d\wedge b)\to (c\vee a), (d\wedge b)\vee (d\wedge b)\to (c\vee a)\bigr) \mid (c,d)\in U\biggr\}\subseteq \cat{LC}(L).$
We first claim that $ A\subseteq U$; since $U$ is an upper set it suffices to show that for any $(c,d)\in U$, one has $(c,d)\sqsubseteq  \bigl((d\wedge b)\to (c\vee a), (d\wedge b)\vee (d\wedge b)\to (c\vee a)\bigr)$, and that is very easy to check.
Hence $A\subseteq U$ and it is a locally exact family because the corresponding join of locally closed sublocales is locally closed. Now, set
$$a_{cd}=(d\wedge b)\to (c\vee a)\quad\text{and}\quad b_{cd}= (d\wedge b)\vee (d\wedge b)\to (c\vee a),$$
$$c_0=  \tbigwedge_{(c,d)\in U} (d\wedge b)\to (c\vee a)$$
and
\begin{align*}d_0&=  \tbigwedge_{x\in L} \biggl(  \big(\tbigwedge _{(c,d)\in U}(b_{cd}\to (x\vee a_{cd})) \big) \to \big( x\vee  \tbigwedge_{(c',d')\in U} a_{c'd'}\bigr)\biggr)\\&=
  \tbigwedge_{x\in L} \biggl(  \big(\tbigwedge _{(c,d)\in U}(b\wedge d)\to (x\vee a_{cd})) \big) \to \big( x\vee  \tbigwedge_{(c',d')\in U} a_{c'd'}\bigr)\biggr)
\end{align*}
(the second equality holds by an application of \ref{H11} and \ref{H2}).  Because $U$ is closed under locally exact families, it follows that $(c_0,d_0)\in U$.

To conclude the proof it suffices to show that $(c_0,d_0)\sqsubseteq (a,b)$. First,
$$c_0=   \tbigwedge_{(c,d)\in U} b\to(d\to(c\vee a)) = b\to \bigl( \tbigwedge_{(c,d)\in U} (d\to(c\vee a))\bigr)=b\to a=a
$$
where we have used \ref{H7}, \ref{H12} and the facts that $a\in\cl(a)\cap\op(b)\subseteq \tbigvee_{(c,d)\in U} \cl(c)\cap\op(d)$ and the frame surjection corresponding to the right hand side is $\tbigwedge_{(c,d)\in U}d\to (c\vee(-))$.

Next, we have to show that $b\leq a\vee d_0$. Note that $a\leq a_{cd}$ for all $(c,d)\in U$ and so $a\leq d_0$. Hence we need to show $b\leq d_0$. Let $x\in L$, we have
\begin{align*} b\wedge  \tbigwedge _{(c,d)\in U}(b\wedge d)\to (x\vee a_{cd})&\leq b\wedge  \tbigwedge _{(c,d)\in U}(b\wedge d)\to ((b\wedge d)\to (x\vee c\vee a)) &
\\& =  b\wedge\tbigwedge _{(c,d)\in U}(b\wedge d)\to  (x\vee c\vee a) & \text{by }\ref{H7}
\\& =  b\wedge\tbigwedge _{(c,d)\in U}b\to ( d\to  (x\vee c\vee a) ) & \text{by }\ref{H7}
\\& = \tbigwedge _{(c,d)\in U}b\wedge ( d\to  (x\vee c\vee a) )& \text{by }\ref{H5}
\\& = b\wedge  \tbigwedge _{(c,d)\in U}  d\to  (x\vee c\vee a) &
\\& = b\wedge( b\to (a\vee x)) &
\\& = b\wedge (a\vee x) & \text{by }\ref{H5}
\\& = a\vee (b\wedge x)\leq(\tbigwedge_{(c',d')\in U}a_{c'd'} )\vee x.&
\end{align*}

where again we have used that the frame surjection corresponding to $\cl(a)\cap\op(b)$ is $\tbigwedge_{(c,d)\in U} d\to (c\vee (-))$.
\end{proof}

\begin{corollary}\label{c: isomorphism SbL and LELCL}
There is an isomorphism $$\SS_b(L)\cong \ca{AU}(\mf{LC}(L)).$$
\end{corollary}

The following diagram summarises the situation (compare the following with diagram \eqref{l: diagram with adjunction for ScL and upsets} for $\Sc(L)$).

\[\begin{tikzcd}
	{\SS_b(L)} & \top & {\mathcal{U}(\mathsf{LC}(L))} \\
	\\
	{\SS_b(L)} && {\mathcal{AU}(\mathsf{LC}(L))}
	\arrow["\varphi", curve={height=-12pt}, from=1-1, to=1-3]
	\arrow["\psi", curve={height=-12pt}, from=1-3, to=1-1]
	\arrow[from=3-1, to=1-1, equal]
	\arrow["\cong"', from=3-1, to=3-3]
	\arrow[hook, from=3-3, to=1-3]
\end{tikzcd}\]

We will now use the characterisation in Corollary~\ref{t: locally exact meet iff exact meet} to describe the closure operator $\varphi\circ \psi$ explicitly. Because admissible upper sets are closed under intersection, they form a closure system, and so for every upper set $U\subseteq A$ there exists the least admissible upper set $\mi{cl}_\ca{A}(U)$ containing it. This is precisely $\varphi(\psi(U))$. Our main result will follow from an explicit description of $\mi{cl}_\ca{A}(-)$ in the general case of the Bruns--Lakser completion of a join-semilattice.

The following is easy to check using routine computations.
\begin{lemma}\label{l: exactness stable under binary joins}
    If $S$ is a join-semilattice, and $\{a_i\}_{i\in I}\subseteq S$ is a family whose meet is admissible, then the meet of $a_i\vee b$ is admissible for every $b\in S$.
\end{lemma}
We also need another technical lemma.
\begin{lemma}\label{l: union of exact families is exact}
    Suppose that $S$ is a join-semilattice and that $\{S_i\}_{i\in I}$ is a collection of admissible families of $S$. Then, if $\{\tbigwedge S_i\mid i\in I\}$ is admissible, so is $\bigcup_i S_i$.
\end{lemma}
\begin{proof}
    Let $a\in S$. We have $\tbigwedge \bigcup_i S_i=\tbigwedge_i \tbigwedge S_i$. As $\{\tbigwedge S_i\}_{i\in I}$ is an admissible family, $(\tbigwedge_i \tbigwedge S_i )\ve a=\tbigwedge_i (\tbigwedge S_i \vee a)$, and as each $S_i$ is admissible, this equals $\tbigwedge_i (\tbigwedge \{s\vee a\mid s\in S_i\})=\tbigwedge \{s\ve a\mid s\in \bigcup_i S_i\}$.
\end{proof}
For $U\se S$ an upper set, we define
\[
\ca{A}(U)=\{\tbigwedge F\mid F\se U,F\text{ is admissible}\}.
\]
The following lemma is essentially proven in \cite[Lemma~3]{BrunsLakser1970}, we report the proof here for the reader's convenience. 

\begin{lemma}\label{l: only one step needed for exact closure}
    If $S$ is a join-semilattice, and $U\se S$ is an upper set, then $\ca{A}(U)$ is the least admissible upper set containing $U$.
\end{lemma}
\begin{proof}
It suffices to show that $\ca{A}(U)$ is an admissible upper set, already. First, we want to show that it is an upper set. Let $x\in \ca{A}(U)$ and $x\leq y$. By definition of $\ca{A}(U)$, there is an admissible family $\{u_i\}_{i\in I}\subseteq U$ such that $x=\tbigwedge_i u_i$. Then, $u_i\vee y\in U$ as $U$ is an upper set, and the family is admissible by Lemma~\ref{l: exactness stable under binary joins}. So, $\tbigwedge_i (u_i\vee y)\in \ca{A}(U)$, but by admissibility of $\tbigwedge_i u_i$ this equals $\tbigwedge_i u_i\vee y$, which is equal to $y$ by our assumptions. So $y\in \ca{A}(U)$. To see that $\ca{A}(U)$ is admissible, suppose that $x_i\in \ca{A}(U)$ and that the meet $\tbigwedge_i x_i$ is admissible. For each $i\in I$, let $S_i\se S$ be the admissible family with $x_i=\tbigwedge S_i$. By Lemma~\ref{l: union of exact families is exact}, $\bigcup_i S_i\se U$ is an admissible family, and so the admissible meet $\tbigwedge \bigcup_i S_i=\tbigwedge_i \tbigwedge S_i$ must be in $\ca{A}(U)$.
\end{proof}

By Corollary~\ref{t: locally exact meet iff exact meet}, and by Lemma~\ref{l: only one step needed for exact closure}, we obtain the following explicit description of $\varphi\circ \psi$.
\begin{proposition}\label{p: what is phipsi explicitly}
    For a frame $L$ and for an upper set $U\se \mf{LC}(L)$, $\varphi(\psi(U))=\ca{A}(U)=\{\tbigsqcap F\mid F\se  U\text{, $F$ is locally exact}\}.$
\end{proposition}

\section{Lifting of morphisms}
\label{sec.lifts}

In this section, we use the characterisation of $\Sb(L)$ as the Bruns--Lakser completion of $\mathsf{LC}(L)$ to give a characterisation of the frame maps which lift to the construction. The following is shown in \cite[Corollary~1]{BrunsLakser1970}.
\begin{lemma}
    For a join-semilattice $S$, the collection $\ca{AU}(S)$ is a frame.
\end{lemma}

\begin{lemma}\label{l: up preserves meet iff exact}\label{l: uparrow turns exact meets into joins}
    If $S$ is a join-semilattice:
    \begin{enumerate}
        \item \label{up preserves all finite joins}$\up:S\to \ca{AU}(S)$ turns all existing joins into meets;
        \item \label{up preserves the lc meets} $\up:S\to \ca{AU}(S)$ turns an existing meet into a join if and only if it is admissible.
    \end{enumerate}
    
\end{lemma}
\begin{proof}
Let us prove the two items in turn.
\begin{enumerate}
    \item For $\{a_i\}_{i\in I}\subseteq S$, suppose that the join $\tbigvee_i a_i\in S$ exists. The meet of $\up a_i$ in $\ca{AU}(S)$ is $\bca_i\up a_i$ (recall that meets in $\ca{AU}(S)$ are intersections). But this is $\up \tbigvee_i a_i$.   
    \item  Suppose that the family $\{a_i\}_{i\in I}\subseteq S$ is admissible. Clearly $\up \tbigwedge_i a_i$ is an upper bound for the family $\up a_i$. If $\up a_i\subseteq E$ for an admissible upper set, $\up \tbigwedge_i a_i\subseteq E$ by admissibility of $\tbigwedge_i a_i$, which shows it is the least upper bound. For the other direction, suppose that there is a family $\{a_i\}_{i\in I}\subseteq S$ whose meet exists and such that $\up \tbigwedge_i a_i$ is the least upper bound of $\{\up a_i\mid i\in I\}$ in $\ca{AU}(S)$. By Lemma~\ref{l: only one step needed for exact closure}, $\up \tbigwedge_i a_i=\ca{A}(\bigcup_i \up a_i)$. Then there is an admissible family $\{x_j\}_{j\in J}\subseteq \bigcup_i \up a_i$ such that $\tbigwedge_j x_j=\tbigwedge_i a_i$. We want to show that $\tbigwedge_i a_i$ is admissible. Let $b\in S$. By our assumption on $x_j$, for every $j\in J$ there is $i\in I$ such that $a_i\ve b\leq x_j\ve b$. Then, $\tbigwedge_i (a_i\ve b)\leq x_j\ve b$ for all $j\in J$. So:
\[
\tbigwedge_i (a_i \ve b)\leq \tbigwedge_j (x_j\ve b)=(\tbigwedge_j x_j)\ve b=(\tbigwedge_i a_i)\ve b.
\]
    Then, the meet $\tbigwedge_i a_i$ is admissible. \qedhere
\end{enumerate}
\end{proof}

A version of the following result is stated as already known in \cite{Gehrke14}. However, we did not find in the literature an explicit proof of this fact. We call a map between join-semilattices $f:S\to T$ \emph{admissible} if it maps admissible families to admissible families, and it also preserves their meets. If a frame map $\ca{AU}(f):\ca{AU}(S)\to \ca{AU}(T)$ extending $f$ exists, since it has to preserve all joins the definition is forced as 
\[
\ca{AU}(f)(U)=\ca{A}(\tbigcup\{\up f(x)\mid x\in U\})
\]
for every admissible upper set $U\se L$. In order for such $\ca{AU}(f)$ to be well defined we need that for all $x\in L$ the following condition holds. 
 \begin{equation*}
x\in \ca{A} (\tbigcup_i U_i)\text{ implies }f(x)\in \ca{A}(\tbigcup_i \up f[U_i]).   \tag{WD}
\end{equation*}
Conversely, if this condition holds, then it is easy to check that $\ca{AU}(f)$ is also a frame map. The existence of a lift is then equivalent to \eqref{e: wd'}. We will use this characterisation several times, in the coming proof and in the later work.

\begin{theorem}\label{t: characterize those that lift general}
    If $f:S\to T$ is a map between join-semilattices, then it lifts to a frame map $\ca{AU}(S)\to \ca{AU}(T)$ if and only if $f$ is admissible.
\end{theorem}
\begin{proof}
Suppose that $f:S\to T$ is a map of join-semilattices. Suppose that $f$ is admissible and $x\in \mathcal{A}(\bigcup_i U_i)$. By Lemma~\ref{l: only one step needed for exact closure}, there is an admissible family $F\se \tbigcup_i U_i$ with $x=\tbigwedge F$. By assumption on $f$, one has  that $\tbigwedge f[F]$ exists and is admissible, and $\tbigwedge f[F]=f(\tbigwedge F)=f(x)$. Since $f[F]\se \tbigcup_i \up f[U_i]$, one has $f(x)\in \ca{A}(\tbigcup_i \up f[U_i])$, as desired. Now suppose that $f$ is not admissible. We will consider three separate cases. In all three cases we will prove that for some admissible family $F$:

\begin{equation}\label{negation of WD'}
\tbigwedge F\in \ca{A} (\tbigcup_{y\in F} \up y)\text{ and }f(\tbigwedge F)\notin \ca{A}(\tbigcup_{y\in F} \up f(y)),
\end{equation}
thus contradicting \eqref{e: wd'}.

\begin{itemize}
    \item Suppose that $F\se S$ is admissible but the meet of the family $f[F]$ does not exist in $T$. By admissibility of $F$, we must have $\tbigwedge F\in\ca{A}(\tbigcup_{y\in F}\up y$). Suppose, towards contradiction, that \eqref{negation of WD'} does not hold, i.e. that there is some admissible family $G\subseteq \tbigcup_{y\in F}\up f(y)$ with $\tbigwedge G=f(\tbigwedge F)$. We will reach a contradiction by deducing that $f[F]$ has a meet, in particular that this is $\tbigwedge G$. This is clearly a lower bound. Suppose that $x\leq f(y)$ for all $y\in F$. For each $g\in G$, there is by assumption $y_g\in F$ with $f(y_g)\leq g$. Then, for each $g\in G$, $x\leq f(y_g)\leq g$, from which $x\leq \tbigwedge G$.
    \item Suppose, now, that $F\se S$ is admissible and the meet of $f[F]$ exists but is not admissible. By Lemma~\ref{l: uparrow turns exact meets into joins}, then, $\up \tbigwedge F=\tbigvee_{y\in F}\up y$, and $\up \tbigwedge f[F]\nsubseteq\tbigvee_{y\in F}\up f(y)$, so $ \tbigwedge f[F]\notin\tbigvee_{y\in F}\up f(y)$. As $f(\tbigwedge F)\leq \tbigwedge f[F]$, then, $f(\tbigwedge F)\notin \tbigvee_{y\in F}\up f(y)$. Then, \eqref{negation of WD'} holds.
    \item Suppose that $F\se S$ is admissible and so is $f[F]$ but $\tbigwedge f[F]\nleq f(\tbigwedge F)$. By Lemma~\ref{l: uparrow turns exact meets into joins}, $\tbigvee_{y\in F}\up f(y)=\ca{A}(\tbigcup_{y\in F}\up f(y))=\up \tbigwedge f[F]$. Then, \eqref{negation of WD'} holds once again.\qedhere
\end{itemize}
\end{proof}

\begin{lemma}\label{l: we link wd with wdb}
For a frame $L$, for $(x,y)\in L\times L$ and $\{(x_i,y_i)\}_{i\in I}\subseteq L\times L$:
\[
\cl(x)\cap \op(y)\se \tbigvee_i \cl(x_i)\cap \op(y_i)\text{ if and only if  }\mi{lc}(x,y)\in \ca{A}(\up \{\mi{lc}(x_i,y_i)\mid i\in I\}).
\]

\end{lemma}
\begin{proof}
   Whenever $(x,y),(x_i,y_i)\in \mf{LC}(L)$, the desired equivalence holds by the isomorphism in Corollary~\ref{c: isomorphism SbL and LELCL}, and by the computations of joins in $\ca{AU}(\mf{LC}(L))$. For arbitrary $(x,y),(x_i,y_i)\in L\times L^{op}$, it suffices to note that for every pair $(a,b)\in L\times L^{op}$, the pair $\mi{lc}(a,b)$ induces the same locally closed sublocale.
\end{proof}

By an argument completely analogous to that before Theorem~\ref{t: characterize those that lift general}, a frame map $f:L\to M$ lifts to a frame map $\overline{f}:\Sb(L)\to \Sb(M)$ if and only if the following well-definedness condition holds for all $x,y\in L$ and $\{x_i\}_{i\in I},\,\{y_i\}_{i\in I}\subseteq L$.

\begin{equation*}
\cl(x)\cap \op(y)\se \tbigvee_i \cl(x_i)\cap \op(y_i) \text{ implies }\cl(f(x))\cap \op(f(y))\se \tbigvee_i \cl(f(x_i))\cap \op(f(y_i)).   \tag{WDb}
\end{equation*}

A frame homomorphism $f\colon L\to M$ will be said to be \emph{locally exact} if the map $\mathsf{LC}(f)\colon \mathsf{LC}(L)\to\mathsf{LC}(M)$ given by $\mathsf{LC}(f)(a,b)=lc(f(a),f(b))$ is admissible. Concretely, $f$ is locally exact if and only if whenever $\{(a_i,b_i)\}_{i\in I}\subseteq \mathsf{LC}(L)$ is locally exact, then $\{ lc(f(a_i), f(b_i) )\}_{i\in I}\subseteq \mathsf{LC}(M)$ is locally exact and 
\begin{equation}\label{adm.mor}\tbigsqcap_{i\in I} lc(f(a_i), f(b_i)) = lc(f\times f)(\tbigsqcap_{i\in I} (a_i,b_i)).\end{equation}

If we set $$S:= \tbigvee_{i\in I} \cl(a_i)\cap \op(b_i)\quad \text{and} \quad T:= \tbigvee_{i\in I} \cl(f(a_i))\cap\op(f(b_i)),$$
expanding the definition of $\mi{lc}$ and using the formula for the meet $\tbigsqcap$, \eqref{adm.mor} means precisely that $$  \tbigwedge T=f(\nu_{S^\#}(\tbigwedge S ))\to\tbigwedge S $$
and $$\nu_{T^{\#}}(\tbigwedge T)=f(\nu_{S^\#}(\tbigwedge S ))\vee \left(f(\nu_{S^\#}(\tbigwedge S ))\to\tbigwedge S \right)$$
(see Lemma~\ref{l: bwe S for a smooth sublocale S} for the explicit computation of $\nu_{S^{\#}}(\tbigwedge S)$).
\begin{theorem}
A frame map $f:L\to M$ lifts to a frame map $\overline{f}:\Sb(L)\to \Sb(M)$ if and only if it is locally exact.
\end{theorem}
\begin{proof}
    A frame map lifts if and only if condition \eqref{e: wdb} holds. But by Lemma~\ref{l: we link wd with wdb}, this is also equivalent to having
    \[
    (x,y)\in \ca{A}(\up \{(x_i,y_i)\mid i\in I\})\text{ implies }\mf{LC}(f)(x,y)\in \ca{A}(\up \{\mf{LC}(f)(x_i,y_i)\mid i\in I\}),
    \]
    using that $\mi{lc}(f(x_i),f(y_i))=\mf{LC}(f)(x_i,y_i)$. This is equivalent to the map $\mf{LC}(f)$ being admissible, as it is a special case of condition \eqref{e: wd'}.
\end{proof}

\end{document}